\documentclass{amsart}%
\usepackage{amsfonts}%
\usepackage{amsmath}%
\usepackage{amssymb}%
\usepackage{graphicx}
\usepackage{float}
\usepackage{enumerate}
\usepackage{hyperref}

\usepackage[T1]{fontenc}

\newtheorem{theorem}{Theorem}
\theoremstyle{plain}

\newtheorem{definition}{Definition}

\newtheorem{lemma}{Lemma}

\newtheorem{problem}{Problem}

\newtheorem{remark}{Remark}

\numberwithin{equation}{section}

\title[Equisingular plane curve singularities]
{A short proof that equisingular plane curve singularities are topologically equivalent}
\author[S. Brzostowski]{Szymon Brzostowski}
\author[T. Krasi\'nski]{Tadeusz Krasi\'{n}ski}
\author[J. Walewska]{Justyna Walewska}

\address{Faculty of Mathematics and Computer Science\newline \indent University of \L\'od\'z\newline  \indent
 Banacha 22, 90-238 \L\'od\'z\newline\indent Poland
}

\email[Szymon Brzostowski]{brzosts@math.uni.lodz.pl}
\email[Tadeusz Krasi\'nski]{krasinsk@uni.lodz.pl}
\email[Justyna Walewska]{walewska@math.uni.lodz.pl}

\begin{document}

\begin{abstract}
We prove that if two plane curve singularities are equisingular, then they are
topologically equivalent. The method we will use is P.~Fortuny~Ayuso's who proved this result for irreducible plane curve singularities.

\end{abstract}
\maketitle

\section{Introduction}

Let $\Gamma,\widetilde{\Gamma}$ be two \textit{plane curve singularities}
(shortly \textit{singularities}) at $0\in\mathbb{C}^{2}.$ We treat a
singularity as the germ of an $1$-dimensional analytic set passing through $0$
or as a representative of such a germ. Among many possible equivalences between
$\Gamma$ and $\widetilde{\Gamma}$: topological, analytic, bilipschitz, etc.\ the most natural is,
in retrospect, the topological one. $\Gamma$ and $\widetilde{\Gamma}$ are
\textit{topologically equivalent} if and only if there exist neighbourhoods
$U_{1}$ and $U_{2}$ of $0\in\mathbb{C}^{2}$ and a homeomorphism $\Phi
:U_{2}\rightarrow U_{1}$ such that $\Gamma\cap U_{1}=\Phi(\widetilde{\Gamma
}\cap U_{2}).$ It is known that the equivalence classes of this relation posses
complete, discrete sets of invariants. Such are, for instance:

\begin{enumerate}[1. ]
	\item the Puiseux characteristic sequences of branches of $\Gamma$ together with
	intersection multiplicities between them,
	
	\item the sequences of multiplicities of branches occurring during the desingularization
	process of $\Gamma$ together with an appropriate relation,
	
	\item the dual weighted graph encoding the desingularization process,
	
	\item the Enriques diagrams,
	
	\item the semigroups of branches and intersection multiplicities between them,
\end{enumerate}

\noindent and many others (see \cite{BK}, \cite{Pfister}, \cite{Wall}). To
establish that each of these data sets is a complete set of topological
invariants of $\Gamma$ is a difficult task. The usual references here are the classical
papers of Brauner \cite{Brauner}, K\"{a}hler \cite{Kahler}, Burau \cite{Burau1}, \cite{Burau2},
but these are difficult to follow and full of the theory of knots (a new approach,
though in the same spirit, can be found in Wall \cite{Wall}, see also \cite{Kras1},
\cite{Kras2}, \cite{Kras3}). In turn, to establish that any two of these data sets
are ``equivalent'', i.e.\ determine one another, is relatively easier. Thus, it
is natural and widely accepted to define: two singularities $\Gamma, $
$\widetilde{\Gamma}$ are \textit{equisingular} if and only if they have the
same data sets of type 1 or 2 or 3 or 4 or 5.

Recently, in the case of branches (=irreducible singularities),  P. Fortuny Ayuso \cite{Ayuso} gave
a new and simple proof of the implication that equisingularity implies topological equivalence,
in which he completely eliminated knot theory. He used only desingularization process.

In the article we extend this result to arbitrary singularities (with many
branches) using the same idea as P.\ Fortuny Ayuso. For a proof of the inverse
implication (also without knot theory) for bilipschitz equivalency see the recent preprint by 
A.~Fernandes, J. E. Sampaio and J.\ P.\ Silva \cite{FernandesSS}. We also
recommend the paper by W.~D.~Neumann and A.~Pichon
\cite{Neumann}. Since Ayuso's method involves desingularization process which itself may be
described in many ways (see the ways 1, 2, 3, 4, 5), we opt for one of these -- the Enriques
diagrams language -- as the most convenient for our purposes.

In Section 2 we recall briefly the desingularization process, the Enriques
diagrams and their properties. Section 3 is devoted to the main result.

\section{Desingularization and the Enriques diagrams}

The basic construction in desingularization process is the blowing-up. Since
desingularization of plane curve singularities leads naturally to blowing-ups
of complex manifolds, we recall this notion right away for manifolds. One
can find the details in many sources \cite{BK}, \cite{Casas}, \cite{Pfister}, \cite{Loj}.

Let $M$ be a 2-dimensional complex manifold and $P\in M.$ The
\textit{blowing-up} of $M$ at $P$ is a 2-dimensional manifold $\widehat{M}$
and a holomorphic mapping $\pi:\widehat{M}\rightarrow M$ with properties:

\begin{enumerate}[1. ]
\item $E:=\pi^{-1}(P)$ is biholomorphic to the 1-dimensional projective space
$\mathbb{P}$ of lines in $\mathbb{C}^{2}$ passing through $0$ $(E$ is called
the \textit{exceptional divisor} of the blowing-up~$\pi),$

\item $\pi|_{\widehat{M}\setminus E}:$ $\widehat{M}\setminus E\rightarrow
M\setminus\{P\}$ is a biholomorphism,

\item for a neighbourhood $U$ of $P$ the mapping $\pi|_{\pi^{-1}(U)}$ is
biholomorphic to a local standard blowing-up of a neighbourhood of
$0\in\mathbb{C}^{2}$, where by the standard blowing-up we mean $\pi^{\operatorname{st}}
:B\rightarrow\mathbb{C}^{2},$ where $B=\{(z,l)\in\mathbb{C}^{2}\times
\mathbb{P}:z\in l\}$ and $\pi^{\operatorname{st}}(z,l):=z.$
\end{enumerate}

The blowing-up $M$ at $P$ always exists and is uniquely defined up to a
biholomorphism. The points in $E$ are called \textit{infinitely near to} $P.$
Since the blowing-up of $M$ at $P$ leads to a manifold $\widehat{M}$, we may repeat the process,
this time blowing-up $\widehat{M}$ at points of $\widehat{M},$ in particular at points in
$E.$ In this case the points in consecutive exceptional divisors are also
called \textit{infinitely near to} $P.$ Since the blowing-up $\pi:\widehat
{M}\rightarrow M$ is a proper mapping, the image $\pi(\Gamma)$ of an arbitrary
analytic subset $\Gamma\subset\widehat{M}$ is an analytic subset of $M.$

Let $\Gamma$ be a plane curve singularity at $P\in M.$ We define the
\textit{proper preimage of} $\Gamma$ as the closure $\overline{\pi^{-1}(\Gamma
\setminus\{P\})}$ and denote it by $\widehat{\Gamma}.$ The analytic set
$\widehat{\Gamma}$ is obtained by adding to the set $\pi^{-1}(\Gamma)\setminus
E$ its accumulation points on $E.$ By an analysis of an equation of
$\widehat{\Gamma}$ in local coordinates in $\widehat{M}$ it follows that the number
of such points is equal to the number of tangent lines to $\Gamma$ at
$P$; in particular, there are only finitely many of them.
All these points are said to \textit{lie on} or \textit{belong to} $\Gamma$. If $Q$ is such a point, then the germ of the proper preimage of $\Gamma$ to
which the point $Q$ belongs is denoted by $\Gamma^{Q}$. In particular, if
$\Gamma$ is an irreducible singularity, then this is just one point (as an irreducible
singularity has only one tangent line). We continue the process of blowing-ups
during desingularization of $\Gamma$ through blowing-ups at consecutive points
which are infinitely near to $P$ and belong to $\Gamma$, until we get a~manifold
$\widehat{\widehat{M}}$ and $\widehat{\pi}:\widehat{\widehat{M}}\rightarrow M$
such that:
\begin{enumerate}[1. ]
\item the proper preimage $\widehat{\widehat{\Gamma}}$ of $\Gamma$ by
$\widehat{\pi}$ is non-singular,

\item in each point of $\widehat{\widehat{\Gamma}}\cap E$ the germ $\widehat
{\widehat{\Gamma}}$ transversally intersects the exceptional divisor $E$ (it means $\widehat{\widehat{\Gamma}}$  and $E$ at such a point are nonsingular and their tangent lines are different).
\end{enumerate}
Two singularities $\Gamma$ and $\widetilde{\Gamma}$ at $0\in\mathbb{C}^{2}$
are equisingular if they  ``have the same desingularization process''. To describe
accurately what it means we have to pay attention to mutual positions of
consecutive proper preimages of the singularity with respect to ``newly pasted''
projective spaces. One of such descriptions is the Enriques' diagram
$E(\Gamma)$ of the resolution of a singularity $\Gamma$. It is a graph
(precisely a tree) with a distinguished root and two kinds of edges: straight
and curved. We outline its construction for an irreducible singularity; for a reducible 
singularity with many branches we construct the Enriques diagrams for each branch
separately and next we ``glue'' these diagrams by identifying vertices representing 
the same infinitely near points and, if neccessary, prolonging blowing-ups to separate branches.

Let $\Gamma$ be an irreducible singularity at $0\in\mathbb{C}^{2}$ and
$\pi:M\rightarrow(\mathbb{C}^{2},0)$ its resolution. The vertices of
$E(\Gamma)$ are all the points belonging to $\Gamma$ in this process of
desingularization (including the point $0\in\mathbb{C}^{2}$ and all points
infinitely near to $0$ that lie on $\Gamma).$ These are centers
of consecutive blowing-ups including also the last one in which the process is
finished. This last point is a maximal point with respect to the partial ordering in the set of points lying on $\Gamma$, induced by successive blowing-ups. The point $0\in\mathbb{C}^{2}$ is a root of $E(\Gamma).$ The edges
of $E(\Gamma)$ connect successive centers. So, for an irreducible singularity
$E(\Gamma)$ is a bamboo (in the language of graph theory). However, there
are two kinds of edges: straight and curved. They are drawn according to the
following rules.

Let $P$ and $Q$ be two successive points that belong to $\Gamma$, and $Q$ is infinitely near to $P$:

\medskip

\begin{enumerate}[1. ]
	\item if $\Gamma^P$ is not tangent to the exceptional divisor at $P$, then edge $PQ$ is curved and moreover
	it  has at $P$ the same tangent as the edge ending at $P$ (Figure \ref{rys:1}).

\begin{figure}[ht]
	\centering
	\includegraphics[width=0.7\textwidth]{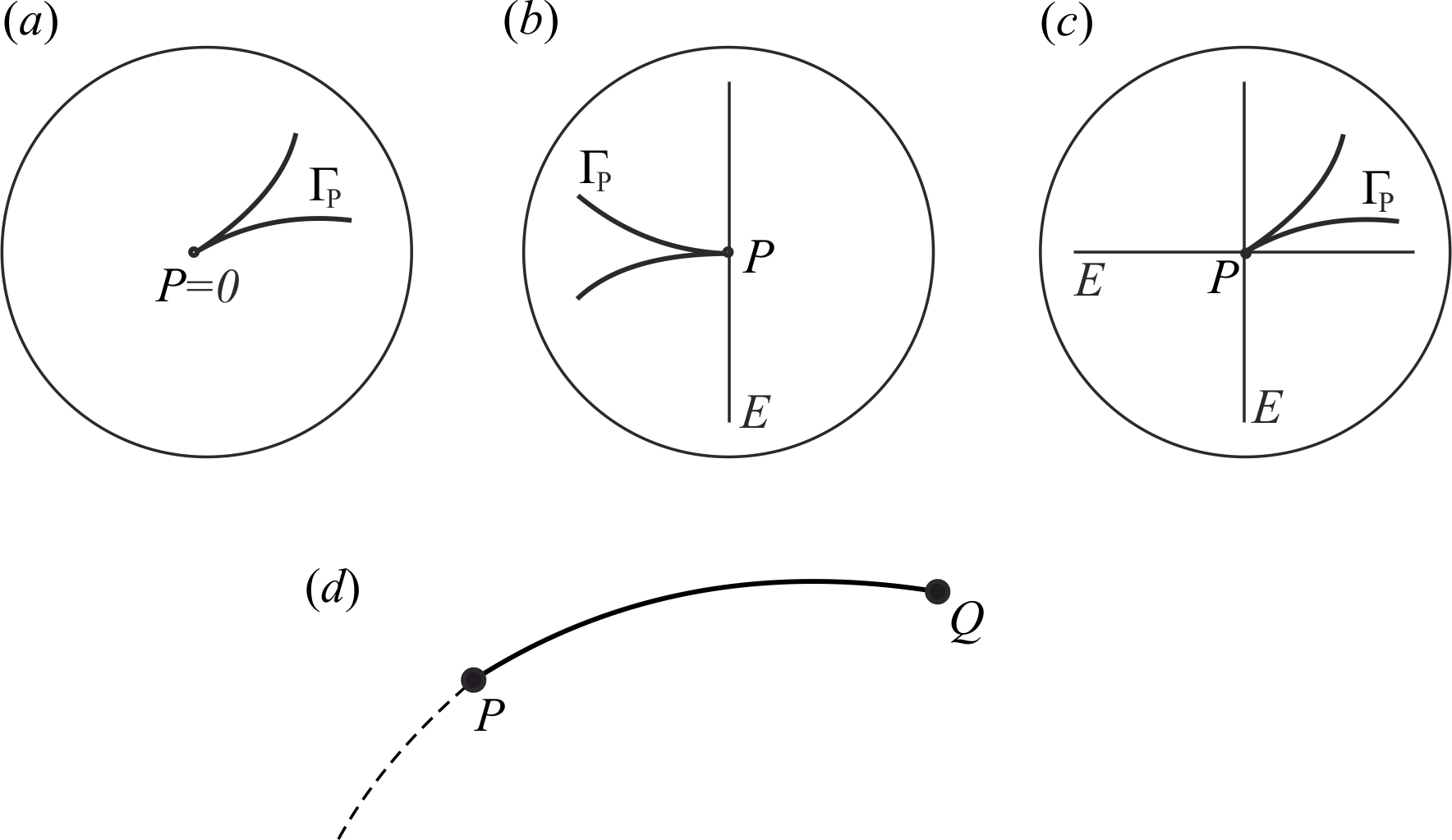}\\
	${}_{}$\\
	\caption{$\Gamma^P$ is not tangent to the exceptional divisor. Possible cases: ($a$) $P=0$ is the root of $E(\Gamma)$, ($b$) $P$ belongs to only one component of $E$, ($c$)~$P$ belongs to two components of $E$, ($d$) corresponding edge in $E(\Gamma)$.}\label{rys:1}
\end{figure}

	\item if $\Gamma^P$ is tangent to the exceptional divisor at $P$, then edge $PQ$ is straight but:

	\begin{enumerate}[a)]
		\item if $\Gamma^P$ is tangent to the ``last-pasted'' projective space, then this straight edge is perpendicular to the edge ending at $P$ (Figure \ref{rys:2}).

\begin{figure}[ht]
	\centering
	\includegraphics[width=0.5\textwidth]{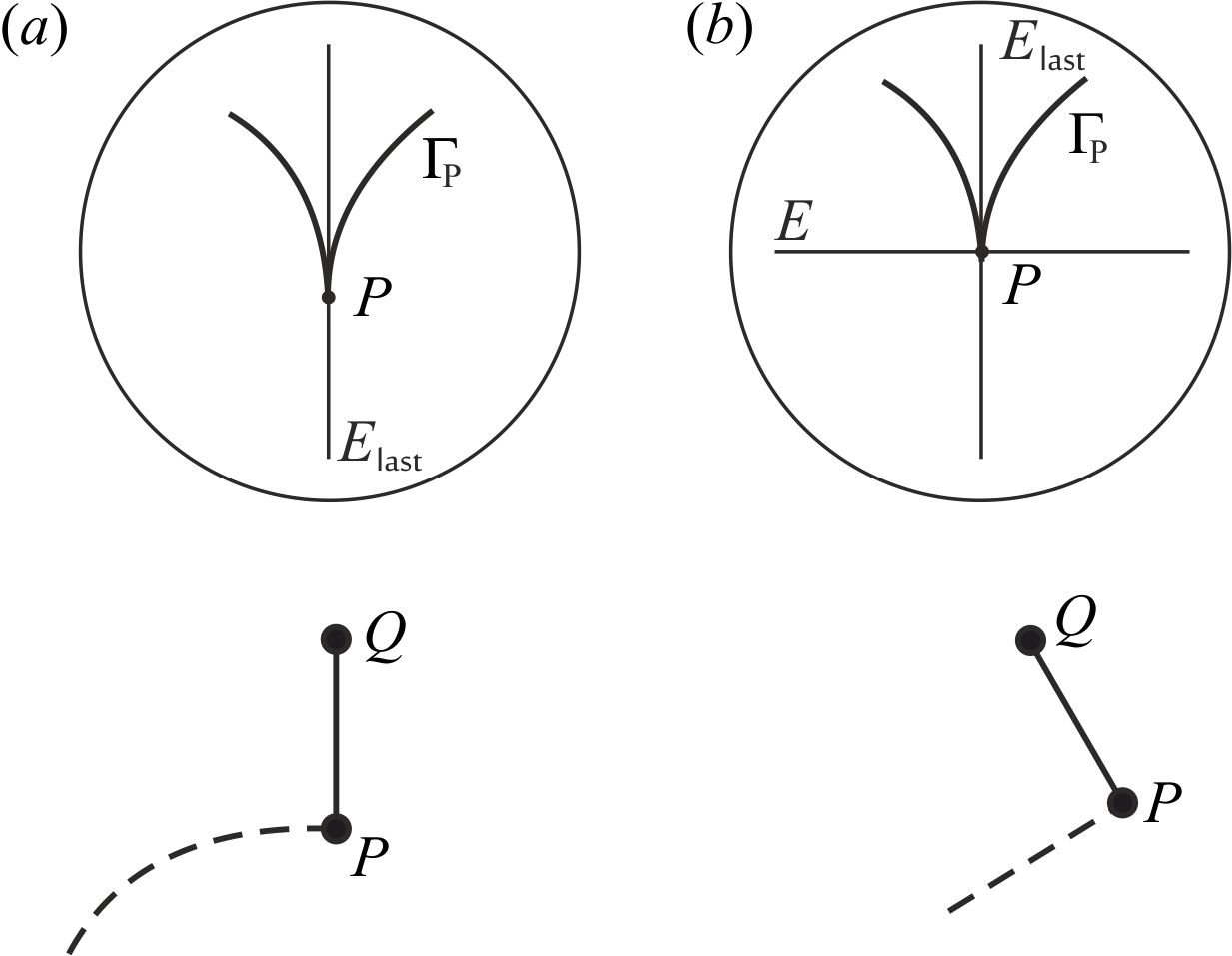}\\
	${}_{}$\\
	\caption{$\Gamma^P$ is tangent to the last-pasted component of $\Gamma$. Possible cases: ($a$)~$P$ belongs to only one component of $E$, ($b$) $P$ belongs to two components of $E$.}\label{rys:2}
\end{figure}

		\item if $\Gamma^P$ is tangent to the ``earlier-pasted'' projective space, then this straight edge is an extension of the previous one, which is also necessarily straight (Figure \ref{rys:3}).
		
		\bigskip
		
		\begin{figure}[ht]
			\centering
			\includegraphics[width=0.2\textwidth]{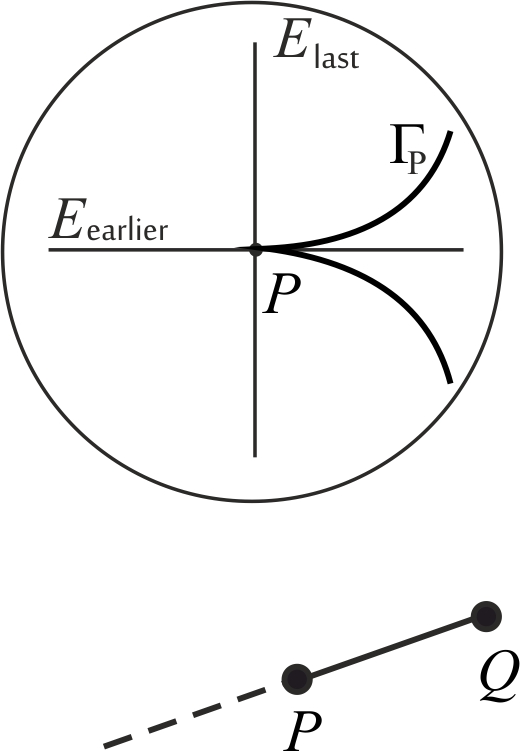}\\
			${}_{}$\\
			\caption{$\Gamma^P$ is tangent to the ``earlier-pasted'' component of $E$.}\label{rys:3}
		\end{figure}	
	\end{enumerate}
\end{enumerate}

\bigskip

The above discussion describes all possible cases that can occur, and thus yields the construction of $E(\Gamma)$.

\begin{remark}
	Notice that, by the very construction of $E(\Gamma)$, both its first and its last edge are always curved.
\end{remark}

\noindent\textbf{Examples} \textbf{ 1. } Let $\Gamma=\{(x,y): x^2-y^3=0\}$. Then $E(\Gamma)$ is as in Figure \ref{rys:4}($a$).

\vspace{0.25cm}

\noindent\textbf{ 2. } Let $\Gamma=\{(x,y): x^2-y^5=0\}$. Then $E(\Gamma)$ is as in Figure \ref{rys:4}($b$).

\vspace{0.25cm}

\noindent\textbf{ 3. } Let $\Gamma=\{(x,y): x^3-y^5=0\}$. Then $E(\Gamma)$ is as in Figure \ref{rys:4}($c$).

\medskip
\begin{figure}[ht]
	\centering
	\includegraphics[width=0.8\textwidth]{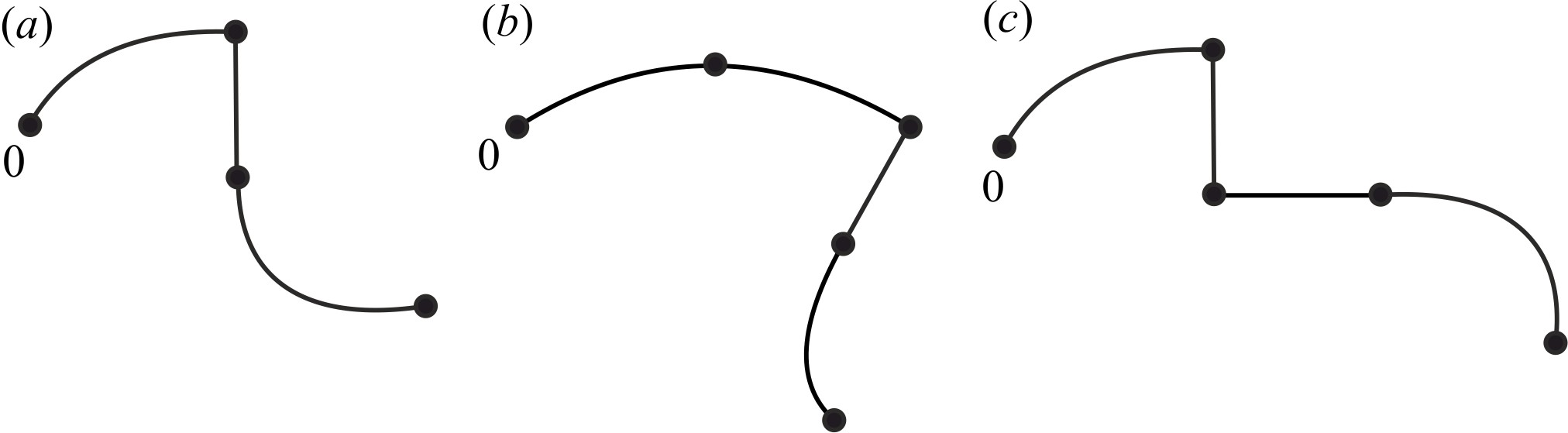}\\
	${}_{}$\\

	\caption{The Enriques diagrams of singularities in Example $1.$}\label{rys:4}
\end{figure}
\medskip

As we stated above, if $\Gamma$ is a reducible singularity with $k$ branches $\Gamma_1,\ldots,\Gamma_k$, then $E(\Gamma)$ is formed in the following way: first we construct $E(\Gamma_1),\ldots,E(\Gamma_k)$ and then we identify all their vertices representing one and the same infinitely near point (in particular, the tree root $0$ is a common point of all $E(\Gamma_i), i=1,\ldots,k$). If the Enriques diagrams of two (or more) branches end at the same infinitely near point, we prolong the process of blowing-ups to separate them. These branches are already non-singular and transversal to the exceptional divisor; so we add only curved edges. It is ilustrated by the following example.

\noindent\textbf{Example} \textbf{ 4. } 
	Let $\Gamma_1=\{(x,y):x^2-y^3=0\}$, $\Gamma_2=\{(x,y):x^2-y^3-y^4=0\}$, $\Gamma=\Gamma_1\cup\Gamma_2=\{(x,y):(x^2-y^3)(x^2-y^3-y^4)=0\}$. Their Enriques diagrams are drawn in Figure \ref{rys:5}.

\vglue.4cm
		
	\medskip

	\begin{figure}[ht]
		\centering
		\includegraphics[width=0.75\textwidth]{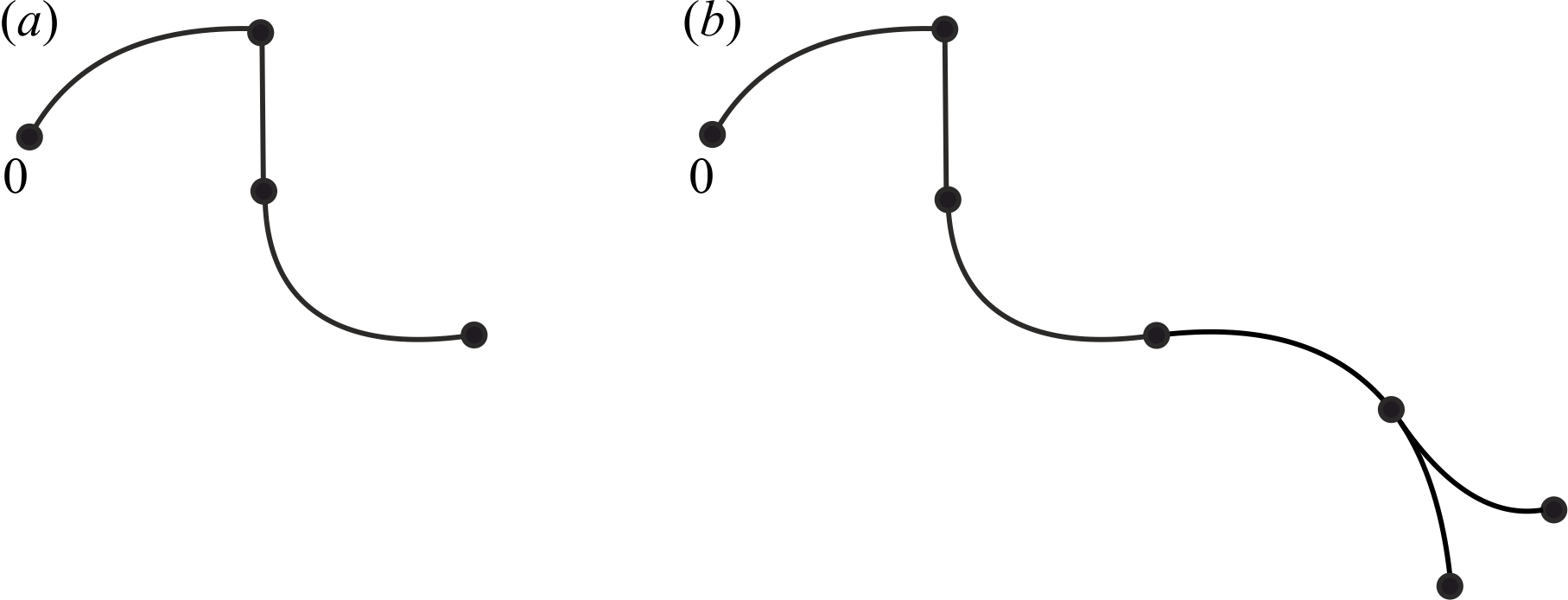}\\
		${}_{}$\\
	
		\caption{($a$) The Enriques diagram of both $\Gamma_1$ and $\Gamma_2$, ($b$) the Enriques diagram of $\Gamma$.}\label{rys:5}
	\end{figure}

It is interesting to observe that $E(\Gamma)$ does not have to be weighted. All necessary data needed to recognize the equisingularity class of $\Gamma$ can be read off from $E(\Gamma)$. In particular, there is a formula for the multiplicity $\mu_P(\Gamma)$ of successive proper preimages of $\Gamma$ (see \cite{Casas}, Theorem 3.5.3) at vertices of $E(\Gamma)$.
Moreover, the intersection multiplicity $\operatorname{i}(\Gamma, \widetilde{\Gamma})$ of two singularities $\Gamma$ and $\widetilde{\Gamma}$ can also be read off from $E(\Gamma)$ and $E(\widetilde{\Gamma})$. This is the famous Noether formula (see \cite{Casas}, Theorem 3.3.1).

\begin{theorem}[Noether's formula]
	If $\Gamma$, $\widetilde{\Gamma}$ are two singularities at $0\in\mathbb{C}^2$, then
	
	$$\operatorname{i}(\Gamma, \widetilde{\Gamma})=\sum_{P\in V(E(\Gamma\cup\widetilde{\Gamma}))}\mu_P(\Gamma^{P})\cdot\mu_P(\widetilde{\Gamma}^{P}),$$
where by $V(E(\Gamma))$ we denote the set of vertices of $E(\Gamma)$.
	
	\end{theorem}

After these preparations, we can finally give a precise definition of equisingularity.
\begin{definition}
	Two plane curve singularities $\Gamma$, $\widetilde{\Gamma}$ are equisingular if their Enriques diagrams $E(\Gamma)$ and $E(\widetilde{\Gamma})$ are isomorphic (it means there exists a graph isomorphism $E(\Gamma)$ with $E(\widetilde{\Gamma})$ which preserves the shapes and angles between edges).
\end{definition}
If $\Gamma$ and $\widetilde{\Gamma}$ are reducible, then the equisingularity of $\Gamma$ to $\widetilde{\Gamma}$ can be equivalently expressed in the terms of their branches and intersection multiplicities (see \cite{Casas}, Theorem 3.8.6).

\begin{theorem}
	If $\Gamma$ has $k$ branches $\gamma_1,\ldots,\gamma_k$ and $\widetilde{\Gamma}$ has $\widetilde{k}$ branches $\widetilde{\gamma}_1, \ldots,\widetilde{\gamma}_{ \tilde{k}}$ then $\Gamma$ and $\widetilde{\Gamma}$ are equisingular if and only if $k=\widetilde{k}$ and, after renumbering branches,
	
	\begin{enumerate}[1. ]
		\item $E(\gamma_i)\cong E(\widetilde{\gamma_i})$, $i=1,\ldots,k$,
		\item $\operatorname{i}(\gamma_i,\gamma_j)=\operatorname{i}(\widetilde{\gamma_i},\widetilde{\gamma_j})$, $i,j=1,\ldots,k$, $i\not=j$.
	\end{enumerate}
\end{theorem}

\section{The main result}

We prove the following theorem

\begin{theorem}
If $\Gamma,\widetilde{\Gamma}$ are two equisingular \textit{plane curve
singularities}, then $\Gamma$ and $\widetilde{\Gamma}$ are topologically equivalent.
\end{theorem}

 First we need several lemmas.

\begin{lemma}
If $\pi:B\rightarrow\mathbb{C}^{2}$ is the standard blowing-up of
$\mathbb{C}^{2}$ at $0$ and $\widetilde{\Phi}:B\rightarrow B$ is a
homeomorphism which keeps the exceptional divisor $E=\pi^{-1}(0)$ invariant
(i.e. $\widetilde{\Phi}(E)=E)$, then the mapping $\Phi:\mathbb{C}%
^{2}\rightarrow\mathbb{C}^{2}$ defined by%
\smallskip
\[
\Phi(x,y):=\left\{
\begin{tabular}
[c]{lll}%
$\pi\circ\widetilde{\Phi}\circ\pi^{-1}(x,y)$ & if & $(x,y)\neq(0,0)$\\
$(0,0)$ & if & $(x,y)=(0,0)$%
\end{tabular}
\right.
\]

is a homeomorphism of $\mathbb{C}^{2}$. We will call $\Phi$ the projection of
$\widetilde{\Phi}.$
\end{lemma}

\begin{proof}
The Lemma is obvious as $\Phi$ is a bijection and $\pi$ is a closed mapping
($\pi$ is even a proper mapping).
\end{proof}

Of course Lemma 1 can be extended to any sequence of blowing-ups.

\begin{lemma}
If $\pi:\widetilde{B}\rightarrow B,$ $B,\widetilde{B}$ -- complex 2-manifolds,
is a composition of blowing-ups and $\widetilde{\Phi}%
:\widetilde{B}\rightarrow\widetilde{B}$ is a homeomorphism which keeps the
exceptional divisor $E$ invariant, then the projection $\Phi$ of $\widetilde
{\Phi}$ is a homeomorphism of $B.$
\end{lemma}

\begin{lemma}
[Ayuso \cite{Ayuso}]\label{Ayuso}Let $\Gamma,\widetilde{\Gamma}$ be two
nonsingular branches transverse to both axes $Ox$ and $Oy.$ Then there exists
a homeomorphism $\Phi:\mathbb{C}^{2}\rightarrow\mathbb{C}^{2}$ of
$\mathbb{C}^{2}$ which is the identity outside any given ball with center at $0$, keeps axes $Ox$ and $Oy$ invariant, is biholomorphic in a neighbourhood of $0\in\mathbb{C}^2$, and $\Phi(\Gamma)=\widetilde{\Gamma}$.
\end{lemma}

\begin{proof}
(Ayuso) We may assume that $\Gamma$ is the germ of the line $L:y=x$ and
$\widetilde{\Gamma}$ is the germ of the parametric curve $y=s(x)=ax+h.o.t.$
with $a\neq0.$ Take the vertical smooth vector field $X=(0,\log\frac{s(x)}%
{x}\cdot y)$ in a sufficiently small neighbourhood of $0\in\mathbb{C}^{2}$, so that
a branch of $\log\frac{s(x)}{x}$ exists, and extend it to a smooth
vertical vector field on the whole of $\mathbb{C}^{2}$ by gluing it with the zero vector field
outside any given ball with center at $0.$ The flow $(\phi_{t})_{t\in
\mathbb{R}}$ for $X$, consisting of diffeomorphisms, is defined for all
$t\in\mathbb{R}$ (because the support of $X$ is compact). The diffeomorphism
$\Phi:=\phi_{1}$ satisfies all required conditions. In fact, since $X$ is
vertical and $X=(0,0)$ on $Ox$, $\Phi$ keeps axes $Ox$ and $Oy$ invariant.
Moreover, for small $x\in\mathbb{C}$
\[
\Phi(x,x)=\phi_{1}(x,x)=c_{(x,x)}(1)=(x,s(x))
\]
where $c_{(x,y)}(t),$ $t\in\mathbb{R},$ is the unique integral curve for $X$
satisfying $c_{(x,y)}(0)=(x,y)$; precisely: $c_{(x,y)}(t)=(x,ye^{t\log s(x)/x})$
for sufficiently small $(x,y)$ and  $t$. Since $X$ is holomorphic in a neighbourhood of
$0$, $\Phi=\phi_{1}$ also is biholomorphic in a neighbourhood of $0.$
\end{proof}

Before we state the next lemma, we introduce a new notion. Let $L:y=ax,$ $a\in\mathbb{C}, $
be a line in $\mathbb{C}^{2}$ and $r>0.$ By a \textit{cone surrounding }$L$
\textit{with radius }$r$ we mean the set $C_{r}(L)$ consisting of all lines
$y=\left(  a+z\right)  x,$ $\left\vert z\right\vert <r$ without the origin. Clearly, 
$C_{r}(L)$ is an open set in $\mathbb{C}^{2}.$

\begin{lemma}
\label{l4}Let $L_{1},\ldots,L_{m}$ and $\widetilde{L}_{1},\ldots,\widetilde
{L}_{m},$ $m\geq1,$ be two systems of different lines in $\mathbb{C}^{2}$ passing
through $0\in\mathbb{C}^{2}.$ Then there exists a homeomorphism $\Phi
:\mathbb{C}^{2}\rightarrow\mathbb{C}^{2}$ of $\mathbb{C}^{2},$ such that:

\begin{enumerate}[1. ]
\item $\Phi$ is the identity outside arbitrary small ball with center at $0, $

\item $\Phi$ transforms the germs of $\widetilde{L}_{i}$ at $0$ onto the germs
of $L_{i}$ at $0,$ for $i=1,\ldots,m,$

\item $\Phi$ transforms biholomorphically some disjoint \textit{cones $\widetilde{C_i}$
surrounding }$\widetilde{L}_{i}$ onto \textit{cones surrounding} $L_{i}$ in a
small neighbourhood of $0$, and each of these biholomorphisms $\Phi|\widetilde{C_i}$ is the
restriction of a biholomorphism of a neighbourhood of $0.$ In the case $\widetilde{L}_i=L_i$ we may choose this biholomorphic restriction to be identity.
\end{enumerate}
\end{lemma}

\begin{proof}
For simplicity we first assume $m=1.$ We may arrange things so that $L:y=ax,$ $\widetilde
{L}:y=bx,$ $a,b\in\mathbb{C},$ $a\cdot b\neq0.$ The linear mapping
$\Psi:\mathbb{C}^{2}\rightarrow\mathbb{C}^{2},$ $\Psi(x,y):=(x,\frac{a}{b}y)$
is a biholomorphism of $\mathbb{C}^{2}$ which transforms $\widetilde{L}$ onto
$L $ and moreover maps any cone $C_{r}(\widetilde{L})$ onto the cone
$C_{r\left\vert a\right\vert /\left\vert b\right\vert }(L).$ We will define
$\Phi$ on each complex plane $\mathbb{C}_{x}:=\{x\}\times\mathbb{C\subset
C}^{2}$ separately. Note that the trace of $C_{r}(\widetilde{L})$ on
$\mathbb{C}_{x}$ is the disk $D(bx,r\left\vert x\right\vert )$ with center at
$bx$ and radius $r\left\vert x\right\vert$, and similarly the trace of
$C_{r\left\vert a\right\vert /\left\vert b\right\vert }(L)$ on $\mathbb{C}_{x}
$ is the disk $D(ax,r\frac{\left\vert a\right\vert }{\left\vert b\right\vert
}\left\vert x\right\vert )$. The restriction $\Psi|\mathbb{C}_{x}$ maps
the disk $D(bx,r\left\vert x\right\vert )$ onto the disk
$D(ax,r\frac{\left\vert a\right\vert }{\left\vert b\right\vert }\left\vert
x\right\vert ).$ Obviously, there exists an extension $\widetilde{\Psi}_{x}$ of
$\Psi|D(bx,r\left\vert x\right\vert )$ to a homeomorphism of the whole
$\mathbb{C}_{x}$ which is identity outside an open ball $D_x$
properly containing both $D(bx,r\left\vert x\right\vert )$ and $D(ax,r\frac{\left\vert
a\right\vert }{\left\vert b\right\vert }\left\vert x\right\vert )$ (i.e.
$\overline{D(bx,r\left\vert x\right\vert )},$ $\overline{D(ax,r\frac
{\left\vert a\right\vert }{\left\vert b\right\vert }\left\vert x\right\vert)
}\subset D$); see Figure \ref{rys:6}.
 
\bigskip

\begin{figure}[ht]
	\centering
	\includegraphics[width=0.8\textwidth]{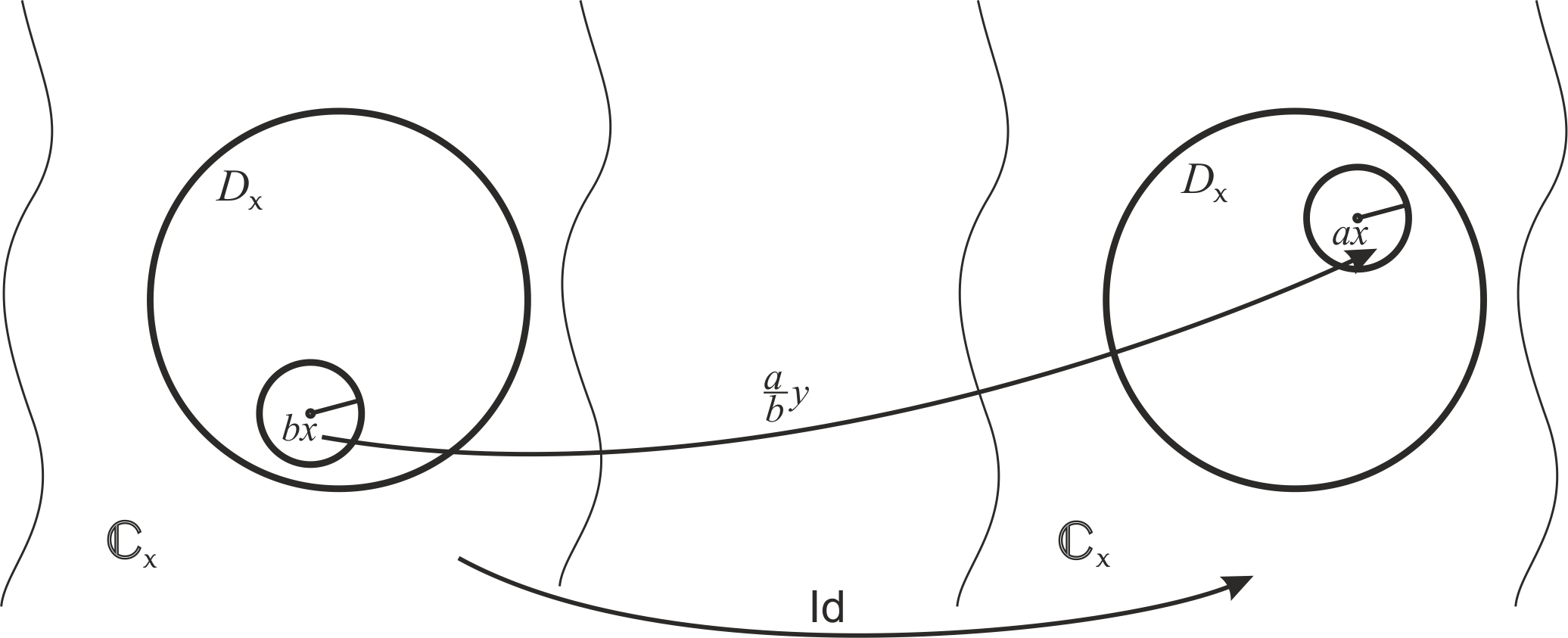}\\

	\caption{A schematic representation of the homeomorphism $\widetilde{\Psi}_x$ on $\mathbb{C}_x$ for small $|x|$.}\label{rys:6}
\end{figure}

Of course, we may choose $D_x$ and $\widetilde{\Psi}_{x}$ so that they also depend continuously on $x.$ Then we
define $\Phi:\mathbb{C}^{2}\rightarrow\mathbb{C}^{2}$ as follows:
\begin{enumerate}[1. ]
\item for small $\left\vert x\right\vert $ we put $\Phi|\mathbb{C}_{x}%
:=\widetilde{\Psi}_{x},$

\item for big $\left\vert x\right\vert $ we put $\Phi|\mathbb{C}_{x}%
:=\operatorname*{Id}|\mathbb{C}_{x},$

\item for intermediate $\left\vert x\right\vert $ we continuously join
$\widetilde{\Psi}_{x}$ to $\operatorname*{Id}.$
\end{enumerate}
The mapping $\Phi$ satisfies all conditions in the assertion of the lemma.

The case $m\geq2$ is similar. We should only choose radii $r$ so that the cones
surrounding $\widetilde{L}_{i}$ and their images under $(x,y)\mapsto
(x,\frac{a_{i}}{b_{i}}y)$ be disjoint.
\end{proof}

Now we may pass to the proof of the main theorem.
\begin{proof}
[Proof of the main theorem]Let $\Gamma,\widetilde{\Gamma}$ be two equisingular
plane curve singularities. Hence their Enriques diagrams $E(\Gamma)$ and
$E(\widetilde{\Gamma})$ are isomorphic. In particular, $\Gamma$ and
$\widetilde{\Gamma}$ have the same number of branches, say $k.$ After
renumbering them we may assume that $\gamma_{1},\ldots,\gamma_{k}$ and
$\widetilde{\gamma}_{1},\ldots,\widetilde{\gamma}_{k}$ are branches of
$\Gamma$ and $\widetilde{\Gamma},$ respectively, and $E(\gamma_{i})\cong
E(\widetilde{\gamma}_{i}),$ $\operatorname{i}(\gamma_{i},\gamma_{j})=\operatorname{i}(\widetilde{\gamma}%
_{i},\widetilde{\gamma}_{j}),$ $i,j=1,\ldots,k,$ $i\neq j.$

The vertices of $E(\Gamma)$ represent points infinitely near to $0\in
\mathbb{C}^{2}$ in the process of desingularization $\pi:B\rightarrow
(\mathbb{C}^{2},0)$ of $\Gamma.$ They are centres of consecutive blowing-ups.
If we apply this process of desingularization $\pi$ to $\widetilde{\Gamma}$,
then some of these points will occur also in $E(\widetilde{\Gamma}).$ For
instance, $0\in\mathbb{C}^{2}$ is a common point of $E(\Gamma)$ and
$E(\widetilde{\Gamma})$ -- it represents the root of $E(\Gamma)$ and
$E(\widetilde{\Gamma}).$ We will prove that $\Gamma$ and $\widetilde{\Gamma}$
are topologically equivalent by induction on the sum $n$ of
numbers of non-common points in $\overline{E}(\gamma_i)$ and $\overline{E}(\widetilde{\gamma}_i)$ for $i=1,\ldots,k$, where $\overline{E}(\gamma_i)$ (respectively  $\overline{E}(\widetilde{\gamma}_i)$) means the subdiagram in $E(\Gamma)$ (resp. $E(\widetilde{\Gamma})$) representing points belonging to $\gamma_i$ (resp. $\widetilde{\gamma_i}$). Notice $E(\gamma_i)$ and $\overline{E}(\gamma_i)$ may differ (see Example 4), but only in points of multiplicity one. 

1.  $n=0.$ This means that the process of desingularization of $\Gamma$ is exactly the
same as of $\widetilde{\Gamma}.$ The centres of consecutive blowing-ups are
exactly the same.

Consider one of the maximal points $P$ of desingularization --
it represents a leaf in $E(\Gamma)$ and simultaneously in $E(\widetilde
{\Gamma}).$ It belongs to only one of branches $\gamma_{1},\ldots
,\gamma_{k}$ and only one of $\widetilde{\gamma}_{1},\ldots,\widetilde{\gamma
}_{k}.$ Since these branches are equisingular, we may assume $P\in\gamma_{1}$
and $P\in\widetilde{\gamma}_{1}$ (see Figure \ref{rys:7}).

\begin{figure}[ht] 
	\centering
	\includegraphics[width=0.28\textwidth]{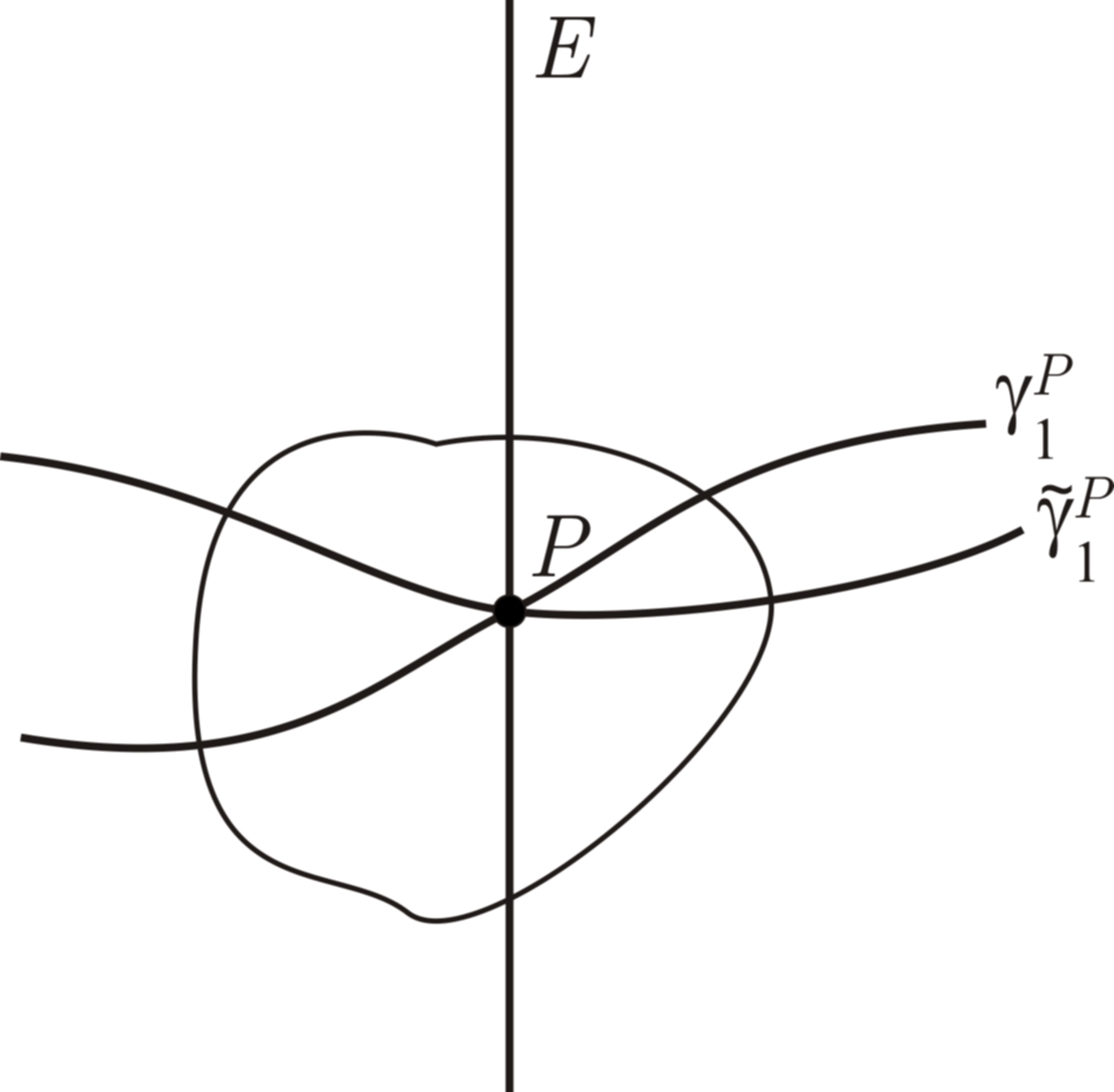}\\
	${}_{}$\\
	
	\caption{The case of nonsingular branches transversal to $E$.}\label{rys:7}
\end{figure}

If we denote the proper preimages of $\gamma_{1}$ and $\widetilde{\gamma}_{1}$
passing through $P$ by $\gamma_{1}^{P}$ and $\widetilde{\gamma}_{1}^{P}$,
 then they are non-singular and transversal to $E$. By Lemma 3, there exists a homeomorphism of $B$ which transforms $\widetilde
{\gamma}_{1}^{P}$ onto $\gamma_{1}^{P}$ in arbitrarily small neighbourhood of $P$,
keeps the exceptional divisor invariant, and is the identity outside another
neighbourhood of $P.$ Doing the same for all maximal points of
desingularization we see that all these homeomorphisms glue to a homeomorphism
of $B$ which transforms the proper preimage of $\widetilde{\Gamma}$ by $\pi$
onto the proper preimage of $\Gamma$ by $\pi$ while keeping the exceptional divisor invariant. By Lemma 2, its projection gives a homeomorphism of $(\mathbb{C}^{2},0)$ which transforms $\widetilde{\Gamma}$ on $\Gamma.$

2.  Assume the theorem holds for any pair of equisingular singularities for
which the number of non-common points in all branches in the
desingularization process of $\Gamma$ is equal to $(n-1),$ $n\geq 1$.

Take now
singularities $\Gamma,\widetilde{\Gamma}$ for which this sum is equal to $n.$
Since $n\geq1$, there exist equisingular branches, say $\gamma_{1}$ and $\widetilde{\gamma
}_{1},$ of $\Gamma$ and $\widetilde{\Gamma}$ such that in $\overline{E}(\gamma_{1})$
there exist points which do not belong to $\overline{E}(\widetilde{\gamma}_{1}).$ Take
the last common point $P$ in $\overline{E}(\gamma_{1})$ and $\overline{E}(\widetilde{\gamma}_{1}).$
In this point the proper preimages $\gamma_{1}^{P}$ and $\widetilde{\gamma}%
_{1}^{P}$ of both branches $\gamma_{1}$ and $\widetilde{\gamma}_{1}$ have different tangent lines $L_1$ and $\widetilde{L}_1$. Moreover,  $\gamma_{1}^{P}$ and $\widetilde{\gamma}_{1}^{P}$ are not tangent to any component of the exceptional divisor (as $\gamma_{1}$ and $\widetilde{\gamma}_{1}$ are equisingular i.e. $E(\gamma_{1})\cong E(\widetilde{\gamma}_{1})$ and in consequence $\overline{E}(\gamma_{1})\cong \overline{E}(\widetilde{\gamma}_{1})$); see Figure \ref{rys:8}.

\begin{figure}[ht]
	\centering
	\includegraphics[width=0.45\textwidth]{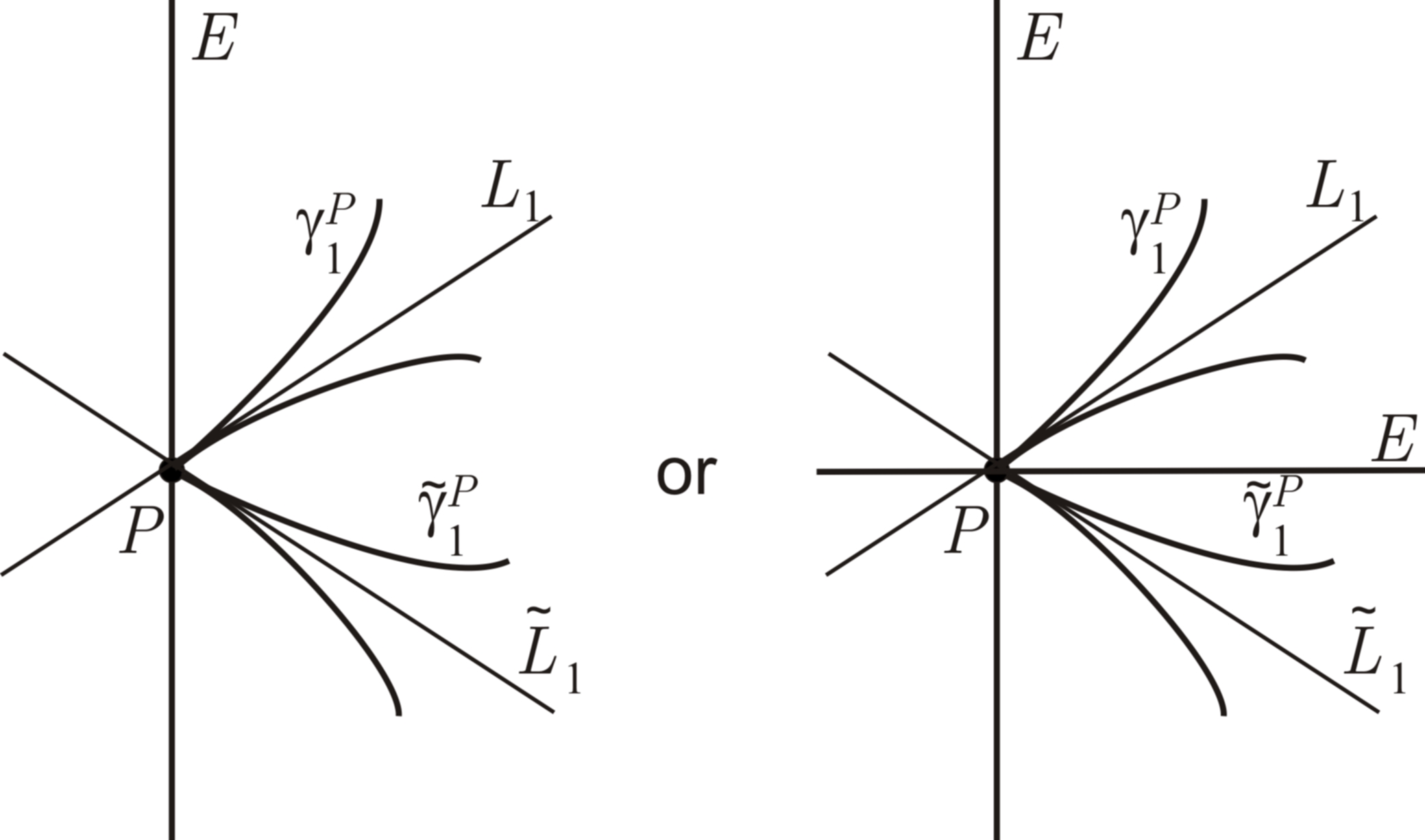}\\
	${}_{}$\\

	\caption{The general case of branches with different tangent lines and not tangent to components of the exceptional divisor.}\label{rys:8}
\end{figure}

It may happen there exist branches of $\Gamma$ whose proper preimages at $P$
have the same tangent line $L_{1}$ as $\gamma_{1}^{P}.$ Assume these are
$\gamma_{2}^{P},\ldots,\gamma_{r}^{P}$, $r\leq k.$ Then, of course,
$\widetilde{\gamma}_{2}^{P},\ldots,\widetilde{\gamma}_{r}^{P}$ share the same
tangent line $\widetilde{L}_{1}$ as $\widetilde{\gamma}_{1}^{P}.$
Moreover,
there may exist other branches of $\Gamma$ whose proper preimages also pass
through $P.$ Assume these are $\gamma_{r+1}^{P},\ldots,\gamma_{s}^{P}$, $s\leq
k$. Their tangent lines are different from $L_1$. Denote all their different tangent lines by $L_2,\ldots,L_m$. Then, by equisingularity of $\gamma_{i}$ to
$\widetilde{\gamma}_{i}$, the branches $\widetilde{\gamma}_{r+1}^{P}%
,\ldots,\widetilde{\gamma}_{s}^{P}$ also pass through $P$ and have also $(m-1)$  different  tangent lines, say $\widetilde{L}_{2}%
,\ldots,\widetilde{L}_{m}$.

Now we apply Lemma \ref{l4} to the manifold $B$
containing $P$ to get a new singularity $\widetilde{\Gamma}^{\prime}$
which will be equisingular and topologically equivalent to $\widetilde{\Gamma}$
and which will have less non-common points with $\Gamma$ in desingularization
process of $\Gamma$ than $\widetilde{\Gamma}$ does. We consider two cases:

(a) among $\widetilde{L}_{2},\ldots,\widetilde{L}_{m}$ there is no $L_{1}$. 
Then in Lemma \ref{l4} we take the systems of lines $\widetilde{L}%
_{1},\ldots,\widetilde{L}_{m}$ and $L_{1},\widetilde{L}_{2},\ldots
,\widetilde{L}_{m}.$ We obtain a homeomorphism $\Phi$ of $B$ which maps
$\widetilde{L}_{1}$ together with branches $\widetilde{\gamma}_{2}^{P}%
,\ldots,\widetilde{\gamma}_{r}^{P}$ tangent to it, respectively, onto $L_{1}$ and some new
branches $\Phi(\widetilde{\gamma}_{1}^{P}),\ldots,\Phi(\widetilde{\gamma}%
_{r}^{P})$ tangent to $L_{1}$, and which leaves the remaining branches
$\widetilde{\gamma}_{r+1}^{P},\ldots,\widetilde{\gamma}_{s}^{P}$ passing
through $P$ unchanged. Moreover, we may assume that $\Phi$ leaves the
exceptional divisor unchanged (in appropriate local coordinates at $P$ the
exceptional divisor may be represented as additional lines in the above systems
of lines). The projections of $\Phi(\widetilde{\gamma}_{1}^{P}),\ldots
,\Phi(\widetilde{\gamma}_{r}^{P})$ to $(\mathbb{C}^{2},0)$ are new branches at
$0\in\mathbb{C}^{2}.$ Denote them by $\widetilde{\gamma}_{1}^{\prime}%
,\ldots,\widetilde{\gamma}_{r}^{\prime}.$ These branches together with
$\widetilde{\gamma}_{r+1}^{\prime}:=\widetilde{\gamma}_{r+1},\ldots
,\widetilde{\gamma}_{k}^{\prime}:=\widetilde{\gamma}_{k}$ define a new plane
curve singularity $\widetilde{\Gamma}^{\prime}.$ We claim $\widetilde{\Gamma
}^{\prime}$ is equisingular and topologically equivalent to $\widetilde{\Gamma
}.$ In fact, regarding equisingularity, we notice $\overline{E}(\widetilde{\gamma}%
_{1}^{\prime}),\ldots,\overline{E}(\widetilde{\gamma}_{r}^{\prime})$ are isomorphic to
$\overline{E}(\widetilde{\gamma}_{1}),\ldots,\overline{E}(\widetilde{\gamma}_{r})$ because
desingularization process up to $P$ is the same for $\widetilde{\gamma}%
_{i}^{\prime}$ and $\widetilde{\gamma}_{i},$ $i=1,\ldots,r$, and at $P$ the
branches $\left(  \widetilde{\gamma}_{1}^{\prime}\right)  ^{P}=\Phi
(\widetilde{\gamma}_{1}^{P}),\ldots,\left(  \widetilde{\gamma}_{r}^{\prime
}\right)  ^{P}=\Phi(\widetilde{\gamma}_{r}^{P})$ and $\widetilde{\gamma}%
_{1}^{P},\ldots,\widetilde{\gamma}_{r}^{P}$ are biholomorphic and not tangent to any components of the exceptional divisior passing through $P$. Hence obviously $E(\widetilde{\gamma}_1^{\prime}),\ldots,E(\widetilde{\gamma}_r^{\prime})$ are isomorphic to $E(\widetilde{\gamma}_1),\ldots,E(\widetilde{\gamma}_r)$. Since
$\widetilde{\gamma}_{r+1}^{\prime}=\widetilde{\gamma}_{r+1},\ldots
,\widetilde{\gamma}_{k}^{\prime}=\widetilde{\gamma}_{k}$, obviously
$E(\widetilde{\gamma}_{r+1}^{\prime})=E(\widetilde{\gamma}_{r+1}%
),\ldots,E(\widetilde{\gamma}_{k}^{\prime})=E(\widetilde{\gamma}_{k}).$
Moreover, the equalities $\operatorname{i}(\widetilde
{\gamma}_{i}^{\prime},\widetilde{\gamma}_{j}^{\prime})=\operatorname{i}(\widetilde{\gamma
}_{i},\widetilde{\gamma}_{j})$ $i,j=1,\ldots,k,$ $i\neq j,$ hold for the same
reasons and because of the Noether's formula. Topological equivalence of $\widetilde{\Gamma
}^{\prime}$ and $\widetilde{\Gamma}$ is obvious as $\Phi$ is a homeomorphism
of $B$ which leaves the exceptional divisor unchanged.

(b) among $\widetilde{L}_{2},\ldots,\widetilde{L}_{m}$ there is $L_{1},$ say
$\widetilde{L}_{2}=L_{1}.$ Then in Lemma \ref{l4} we take the systems of lines
$\widetilde{L}_{1},\widetilde{L}_{2},\widetilde{L}_{3}\ldots,\widetilde{L}%
_{m}$ and $L_{1},L^{\prime},\widetilde{L}_{3},\ldots,\widetilde{L}_{m},$ where
$L^{\prime}$ is a new line different from $L_{1},\widetilde{L}_{3}%
,\ldots,\widetilde{L}_{m}.$ The same reasoning as in item (a) also gives a new
singularity $\widetilde{\Gamma}^{\prime}$ which is equisingular and
topologically equivalent to $\widetilde{\Gamma}$ and which has less non-common
points with $\Gamma$ in desingularization process of $\Gamma$ than
$\widetilde{\Gamma}$ does.

In each case we get $\widetilde{\Gamma}^{\prime}$ which is equisingular to
$\widetilde{\Gamma}$, and hence to $\Gamma$, and which has less non-common
points with $\Gamma$ in desingularization process of $\Gamma$ than
$\widetilde{\Gamma}$ does. By induction hypothesis, $\Gamma$ is topologically
equivalent to $\widetilde{\Gamma}^{\prime}$ and hence to $\widetilde{\Gamma}.$
This ends the proof.
\end{proof}

\begin{problem}
	As we know the topological equivalence of plane curve singularities is the same as their bilipschitz equivalence \cite{Neumann}, we pose the problem to find, using the Ayuso's method, a bilipschitz homeomorphism.
\end{problem}


\begin{thebibliography}{99}
\bibitem{Brauner} K. Brauner, \emph{Das Verhalten der Funktionen in der Umgebung ihrer Verzweigungsstellen}, Abh. Math. Sem. Univ. Hamburg 6(1) (1928), 1--55.

\bibitem{BK} E. Brieskorn and H. Kn\"orrer, \emph{Plane algebraic curves}, Birkh\"auser Verlag, Basel, 1986. Translated from the German by John Stillwell.

\bibitem{Burau1} W. Burau, \emph{Kennzeichnung der Schlauchknoten}, Abh. Math. Sem. Univ. Hamburg 9(1) (1933), 125--133.

\bibitem{Burau2} W. Burau, \emph{Kennzeichnung der Schlauchverkettungen}, Abh. Math. Sem. Univ. Hamburg,
10(1) (1934), 285--297.

\bibitem{Casas} E. Casas-Alvero, \emph{Singularities of plane curves}, volume 276 of London Mathematical Society
Lecture Note Series. Cambridge University Press, Cambridge, 2000.

\bibitem{Pfister} T. de Jong and G. Pfister, \emph{Local analytic geometry}, Advanced Lectures in Mathematics.
Friedr. Vieweg \& Sohn, Braunschweig, 2000, Basic theory and applications.

\bibitem{FernandesSS} A. Fernandes, J. E. Sampaio, and J. P. Silva, \emph{H\"older equivalence of complex analytic curve singularities}, ArXiv e-prints, April 2017.

\bibitem{Ayuso} P. Fortuny Ayuso, \emph{A short proof that equisingular branches are isotopic}, ArXiv e-prints,
March 2017.

\bibitem{Kahler} E. K\"ahler, \emph{\"Uber die Verzweigung einer algebraischen Funktion zweier Ver\"anderlichen in der
Umgebung einer singul\"aren Stelle}, Math. Z. 30(1) (1929), 188--204.

\bibitem{Kras1} T. Krasi\'nski, \emph{Curves and knots I. Torus knots of first order}, In XXXII Conference and
Workshop ``Analytic and Algebraic Geometry'',  23--45. University of {\L}\'od\'z Press, 2011,
(in Polish).

\bibitem{Kras2} T. Krasi\'nski, \emph{Curves and knots II. Torus knots of higher order}, In XXXIII Conference and Workshop ``Analytic and Algebraic Geometry'',  33--49. University of {\L}\'od\'z Press, 2012,
(in Polish).

\bibitem{Kras3} T. Krasi\'nski, \emph{Curves and knots III. Knots of analytic irreducible curves}, In XXXIV Con-
ference and Workshop ``Analytic and Algebraic Geometry'',  15--25. University of {\L}\'od\'z Press, 2013, (in Polish).

\bibitem{Loj} S. {\L}ojasiewicz, \emph{Geometric desingularization of curves in manifolds}, In Analytic and Algebraic Geometry,  11--32. Faculty of Mathematics and Computer Science. University of {\L}\'od\'z, 2013. Translated from the Polish by T. Krasi\'nski.

\bibitem{Neumann} W. D. Neumann and A. Pichon, \emph{Lipschitz geometry of complex curves}, J. Singul., 10 (2014), 225--234.
\bibitem{Wall} C. T. C. Wall, \emph{Singular points of plane curves}, volume 63 of London Mathematical Society
Student Texts. Cambridge University Press, Cambridge, 2004.
\end{thebibliography}
\end{document}